\documentclass[a4paper,twoside]{article}      
\usepackage{amsmath,amssymb,amsfonts,amsthm,amscd} 
\usepackage{graphics}                 
\usepackage{color}                    
\usepackage{hyperref}                
\usepackage{ mathrsfs }
\usepackage{indentfirst}
\usepackage{bbold}
\usepackage{enumerate}
\usepackage{url}         
\usepackage{colonequals} 
\usepackage{a4wide}

\newtheorem{theorem}{Theorem}[section]
\newtheorem{proposition}[theorem]{Proposition}

\newtheorem{lemma}[theorem]{Lemma}
\theoremstyle{definition}
\newtheorem{remark}[theorem]{Remark}

\newtheorem{assumption}{Assumption}[section]

\def\div{\mathop{\mathrm{div}}\nolimits}
\def\!{\mathop{\mathrm{!}}}

\def\R{\mathbb{ R}}

\newlength{\boxwidth}
\date \today
\title{Comparison and maximum principles for a class of flux-limited diffusions with external force fields}
\author{Manh Hong Duong \footnote{Mathematics Institute, University of Warwick, Coventry CV4 7Al, UK. Email: m.h.duong@warwick.ac.uk}}
\begin{document}
\maketitle
\begin{abstract}
In this paper, we are interested in a general equation that has finite speed of propagation compatible with Einstein's theory of special relativity. This equation without external force fields has been derived recently by means of optimal transportation theory. We first provide an argument to incorporate the external force fields. Then we are concerned with comparison and maximum principles for this equation. We consider both stationary and evolutionary problems. We show that the former satisfies a comparison principle and a strong maximum principle. While the latter fulfils weaker ones. The key technique is a transformation that matches well with the gradient flow structure of the equation.
\end{abstract}
\section{Introduction}
\subsection{Motivation}
This paper is concerned with comparison and maximum principles for the following equation
\begin{equation}
\label{eq: gRHE eqn}
\partial_t u=\div[u\nabla\varphi^*(\nabla V)]+\div[u\nabla \varphi^*(\nabla \log u)],
\end{equation}
In this equation, the spatial domain is $\R^d$; the unknown is $u:[0,T]\to \R^d$; $\div$ denotes the divergence operator; the function $V: \R^d\to \R$ is given, and finally $\varphi^*$ is the Legendre transformation of a given function $\varphi\in C^2(\R^d)$. A typical example of Eq. \eqref{eq: gRHE eqn} is when $\varphi$ is the so-called relativistic cost function
\begin{equation}
\label{eq: rel cost}
\varphi(x)\equiv\varphi_c(x)=\begin{cases}
\left(1-\sqrt{1-\frac{|x|^2}{c^2}}\right)c^2\quad \text{if}\quad |x|\leq c,\\
+\infty\quad \text{if}\quad |x|> c,
\end{cases}
\end{equation}
where $c$ is a given constant which has the meaning of the speed of light. In this case,
\begin{equation*}
\varphi_c^*(x)=c^2\left(\sqrt{1+\frac{|x|^2}{c^2}}-1\right),\quad\text{and}	\quad \nabla \varphi_c^*(x)=\frac{x}{\sqrt{1+\frac{|x|^2}{c^2}}},
\end{equation*}
and Eq. \eqref{eq: gRHE eqn} becomes the well-known relativistic heat equation (with external force fields)
\begin{equation}
\label{rel heat equation}
\partial_t u=\div\left(\frac{u\nabla V}{\sqrt{1+\frac{|\nabla V|^2}{c^2}}}\right)+\div\left(\frac{u\nabla u}{\sqrt{u^2+\frac{|\nabla u|^2}{c^2}}}\right).
\end{equation}
In general, in this paper, the function $\varphi$ will be of general form
\begin{equation}
\label{eq: general cost}
\varphi(x)=\begin{cases}
\tilde{c}(|x|),\quad \text{if}\quad |x|\leq c,\\
+\infty\quad \text{if}\quad |x|> c,
\end{cases}
\end{equation}
where $\tilde{c}$ is a given function on the real line.

Let us first review relevant literature about the origin of Eq. \eqref{eq: gRHE eqn}. It is well-known that the classical heat equation, $\partial_t u=\Delta u$, satisfies a strong maximum principle which implies infinite speed of propagation. However this is incompatible with Einstein's theory of special relativity that the speed of light is the highest admissible free velocity in a medium. Eq. \eqref{eq: gRHE eqn}, which belongs to a larger class of the so-called flux-limited diffusions, have recently been introduced to fulfil Einstein's theory. In particular, the relativistic heat equation (i.e., Eq. \eqref{rel heat equation} without external field $V=0$) was proposed first by Rosenau~\cite{Ros92} and was derived later by Brenier~\cite{Bre03} by means of optimal transportation theory. He formally showed that the relativistic heat equation is a gradient flow of the Boltzmann entropy with respect to a Kantorovich optimal transport functional associated to the relativistic cost \eqref{eq: rel cost}. In \cite{CP09}, the authors proved this rigorously and generalized to the general cost in \eqref{eq: general cost} obtaining Eq. \eqref{eq: gRHE eqn} with $V=0$. Eq. \eqref{rel heat equation} with a non-zero force field has also appeared in other contexts. For instance, to be able to describe the velocity profile of charge carriers along submicrometer electronic devices under high fields, in \cite{Zakari97}, the author obtained Eq. \eqref{rel heat equation} using the method of continued-fraction expansion. As another example, in the context of the flux-limited chemotaxis model, in \cite{BBNS10} the author derived a similar equation as \eqref{rel heat equation} by optimizing the the flux density of particles along the trajectory induced by the chemoattractant. Eq. \eqref{eq: gRHE eqn} with a non-zero external force field is a generalisation of Eq. \eqref{rel heat equation} to the general cost function \eqref{eq: general cost}.


Eq. \eqref{eq: gRHE eqn}, and in general flux-limited diffusion equations, play an important role both in theoretical research and in practical applications. From a practical point of view, these models have been used widely, for instance, in plasma physics~\cite{ACM05a, ACM05b}, in the modelling of the charge transport in microelectronic devices~\cite{Han91, ZJ98}, and of multicellular growing systems~\cite{BBNS10, BBNS12}. Besides their practical usage, they also create interesting mathematical problems due to the nonlinearity and discontinuity of the function $\varphi$. To solve these problems, many authors have established intriguing connections between different fields of mathematics such as nonlinear partial differential equations, optimal transportation and differential geometry, see e.g.,\cite{Bre03, ACM06, CP09,BP13} and references therein.

\subsection{Aim of the paper}
The aim of this paper is to study comparison and maximum principles for Eq. \eqref{eq: gRHE eqn}. Comparison and maximum principles have been the subjects of intensive research, see the monographs~\cite{Evans98, GT01, PS07} and references therein for more information. These principles provide insightful information about partial differential equations such as propagation of speed and uniqueness.
As mentioned  in the previous section, Eq. \eqref{eq: gRHE eqn} has finite speed of propagation; hence it certainly does not satisfy a strong maximum principle. Therefore, it is natural to ask whether it does fulfil a weak maximum principle. This is the focus of the present paper. We consider both the stationary and evolutionary problems. We show that the former satisfies a comparison principle and a strong maximum principle. While the latter fulfils weaker ones. It should be mentioned that in a recent paper~\cite{MS15TR}, the authors address this issue affirmatively for the case of the relativistic heat equation (i.e. Eq. \eqref{rel heat equation} with $V\equiv 0$). The present paper extends their results to Eq. \eqref{eq: gRHE eqn} with non-zero external force fields. As presented above, Eq. \eqref{eq: gRHE eqn} covers a much larger class of equations that are of importance in both theoretical research and practical applications.

We now summarize our assumptions and main results.

\subsection{Assumptions and main results}
\begin{assumption} Throughout the paper, we make the following assumptions.
\begin{enumerate}[(H1)]
\item $V\in C^2(\R^d)$,
\item $\varphi^*\in C^2(\R^d)$ is strictly convex (i.e., all eigenvalue of the Hessian $\nabla^2\varphi^*$ are positive). 
\end{enumerate}
\end{assumption}
We now describe our main results. Let $Lu$ denote the operator on the right hand side of Eq. \eqref{eq: gRHE eqn}, i.e.,
\begin{equation}
\label{eq: operator L}
Lu=\div[u\nabla\varphi^*(\nabla V)]+\div[u\nabla \varphi^*(\nabla \log u)].
\end{equation}
We consider both the stationary problem
\begin{equation}
\label{eq: stationary problem}
Lu=0,
\end{equation}
and the evolutionary one
\begin{equation}
\label{eq: evolutionary prob}
\partial_t u=Lu.
\end{equation}
The first theorem is a tangency principle~\cite{PS07} for the stationary problem.
\begin{theorem}[Tangency principle for the stationary problem]
\label{theo: tangency L}
\ \\
Suppose that $\Omega\subset \R^d$ is open and connected. Assume that $u, v\in C^2(\Omega)$ and $u,v>0$. If $Lu\geq Lv$ and $u\leq v$ in $\Omega$ and $u=v$ at some point in $\Omega$, then $u=v$ in $\Omega$.
\end{theorem}
We obtain the following strong maximum principle as a direct consequence of the tangency principle.
\begin{theorem}[Strong maximum principle for the stationary problem]
\label{theo: strong max L}
\ \\
Suppose that $\Omega\subset \R^d$ is open and connected. Assume $u\in C^2(\Omega)\cap C(\bar{\Omega})$ and $u>0$. 
\begin{enumerate}[(1)]
\item Additionally, assume that $\div(\nabla\varphi^*(\nabla V))\leq 0$. If $Lu\geq 0$ in $\Omega$ and $u$ attains its maximum over $\bar{\Omega}$ at an interior point, then $u$ is constant within $\Omega$. 
\item Similarly, assume in addition that $\div(\nabla\varphi^*(\nabla V))\geq 0$. If $Lu\leq 0$ in $U$ and $u$ attains its minimum over $\bar{\Omega}$ at an interior point, then $u$ is constant within $\Omega$. 
\end{enumerate}
\end{theorem}
A comparison principle also holds true for the stationary problem.
\begin{theorem}[Comparison principle for the stationary problem]
\ \\
\label{theo: comparison L}
Suppose $u,v\in C^2(\Omega)\cap C(\bar{\Omega})$ with $u,v>0$ and satisfy $Lu\geq Lv$ in $\Omega$. If $u\leq v$ on $\partial \Omega$, then $u\leq v$ in $\Omega$.
\end{theorem}
Next we deal with the evolutionary case. Let $\Omega\subset \R^d$ be an open set and $T>0$ be given. We denote the space-time domain by $\Omega_T:=\Omega\times [0,T]$. The boundary of $\Omega_T$ is $=\overline{\Omega_T}\setminus \Omega_T$. Finally, $C^{2,1}(\Omega_T)$ denotes the set of functions of $(x,t)\in\Omega_T$ which is $C^2$ in space, and $C^1$ in time. 
\begin{theorem}[Comparison principle for the evolutionary problem]
\ \\
\label{theo: comparison L evol}
Suppose that $u,v\in C^{2,1}(\Omega_T)\cap C(\bar{\Omega_T}), u,v>0$ satisfying $\partial_t u-L u\leq \partial_t v-Lv$ in $\Omega$. Then $\max\limits_{\overline{\Omega_T}}(u-v)=\max\limits_{\Gamma_T}(u-v)$. In particular, if $u\leq v$ on $\Gamma_T$, then $u\leq v$ in $\Omega_T$.
\end{theorem}
As a consequence of the comparison principle we obtain the weak maximum principle for evolutionary problem.
\begin{theorem}[Weak maximum principle for the evolutionary problem]
\label{theo: weak max L evol}
\ \\
Suppose that $u\in C^{2,1}(\Omega_T)\cap C(\overline{\Omega_T})$ and $u>0$. The following statement hold.
\begin{enumerate}
\item Additionally, assume that $\div(\nabla\varphi^*(\nabla V))\leq 0$. If $\partial_t u-Lu\leq 0$, then a maximum of $u$ in $\overline{\Omega_T}$ is attained at the boundary $\Gamma_T$.
\item Similarly, assume in addition that $\div(\nabla\varphi^*(\nabla V))\geq 0$. If $\partial_t u-Lu\geq 0$, then a minimum of $u$ in $\overline{\Omega_T}$ is attained at the boundary $\Gamma_T$.
\end{enumerate}
\end{theorem}
The underlying idea of the proofs of the these theorems is to transform
the operator $L$ into a strictly elliptic operator $Q$ that is independent of the function itself. The assertions for $Q$ are deduced from well-known results on maximum/comparison principle for quasi-nonlinear partial differential equations in~\cite{Evans98,GT01,PS07}. Then we obtain the corresponding results for the original operator $L$ using the monotonicity of the transformation. 
%

\subsection{Organisation of the paper}
The rest of the paper is organized as follows. In Section~\ref{sec: derivation}, we summarize the derivation of the main equation. In Section \ref{sec: eliptic}, we prove the main theorems for the stationary problem. Section \ref{sec: parabolic} is devoted to the evolutionary case. Finally, we conclude the paper and discuss future perspectives in Section \ref{sec: conclusion}.
\section{On the derivation of the main equation}
\label{sec: derivation}
In this section, we derive Eq. \eqref{eq: gRHE eqn} and prove that, under an appropriate assumption on the growth of $V$, it has an unique equilibrium solution given by the Gibbs measure.

In a seminal paper~\cite{JKO98}, Jordan, Kinderleher and Otto proved a remarkable result that the classical heat equation
\begin{equation}
\label{heat eqn}
\partial_t u=\Delta u
\end{equation}
is a gradient flow of the Boltzmann entropy $S(u)=\int (u\log u-u)$ with respect to the Wasserstein metric. This statement means that the heat equation can be obtained by computing the time discrete solution at the time step $k$ as follows (now known as JKO-scheme)
\begin{equation}
\label{eq: JKO scheme}
u^k=\mathrm{argmin}_{u\in \mathcal{P}(\R^d)} S(u)+hW_{2,h}^2(u^{k-1},u),
\end{equation}
where the Wasserstein distance $W_{2,h}$ is defined via
\begin{equation*}
W_{2,h}^2(u^{k-1},u)=\inf_{\gamma\in \Gamma(u^{k-1},u)}\int_{\R^d\times\R^d}\Big(\frac{|x-y|}{h}\Big)^2\gamma(dx,dy),
\end{equation*}
and passing to the limit as the time step $h$ goes to $0$. In the expressions above, $\mathcal{P}(\R^d)$ is the set of all probability measures on $\R^d$ and $\Gamma(\mu,\nu)$ denotes the set of probability measures on $\R^{2d}$ whose marginals are $\mu$ and $\nu$. To obtain the linear drift-diffusion equation
\begin{equation}
\label{eq: FP eqn}
\partial_tu=\div(u\nabla V)+\Delta u,
\end{equation}  
one simply replaces the Boltzmann entropy $S$ by the free energy $F=S+\int V u$.
Since then the paper \cite{JKO98} has embarked a lot of research that links many branches of mathematics together such as partial differential equation, optimal transportation, differential geometry and probability theory, see e.g., \cite{ADPZ11, DLR13, DPZ13a, DPZ13b}, the monographs \cite{AGS08, Vil03} and references therein for more information.

In \cite{Bre03}, Brenier formally derived the relativistic heat equation (i.e., Eq. \eqref{rel heat equation} with $V=0$), from the scheme \eqref{eq: JKO scheme} replacing $W_{2,h}^2$ by
\begin{equation}
W_c=\inf_{\gamma\in \Gamma(u^{k-1},u)}\int_{\R^d\times\R^d}\varphi_c\Big(\frac{x-y}{h}\Big)\gamma(dx,dy),
\end{equation}
where $\varphi_c$ is the relativistic cost function \eqref{eq: rel cost}.
Later on, McCann and Puel~\cite{CP09} proved this rigorously and extended to the general case with $\varphi$ given in \eqref{eq: general cost}, obtaining the following equation
\begin{equation}
\label{eq: eqn without force field}
\partial_t u=\div[u\nabla \varphi^*(\nabla \log u)],
\end{equation} 
which is Eq. \eqref{eq: gRHE eqn} where the external force is absent. 

In contrast to the linear drift-diffusion equation \eqref{eq: FP eqn}, it is not obvious how to incorporate the external force fields into Eq.~\eqref{eq: eqn without force field}. One possibility, as in the classical setting, is to replace the Boltzmann entropy $S=\int (u\log u-u)$ by the free energy $F=S+\int Vu$. This amounts to substitute $\frac{\delta S}{\delta u}=\log u$ in \eqref{eq: eqn without force field} by $\frac{\delta F}{\delta u}=\log u+\nabla V$ leading to the following equation
\begin{equation*}
\partial_t u=\div[u\nabla \varphi^*(\log u+ \nabla V)].
\end{equation*} 
For the relativistic cost, since $\nabla \varphi_c^*(z)=\frac{z}{\sqrt{1+\frac{|z|^2}{c^2}}}$ the above equation becomes
\begin{equation}
\partial_t u=\div\left[u\frac{\log u+ \nabla V}{\sqrt{1+\frac{|\log u+ \nabla V|^2}{c^2}}}\right].
\end{equation}
Note that due to the nonlinearity of $\nabla \varphi^*_c(z)$, this equation is different from Eq. \eqref{rel heat equation}. In comparison with the classical equation \eqref{eq: FP eqn}, the drift and the diffusion terms are separated since $\nabla\varphi^*(z)=z$ is linear. We do not know whether the equation above has any physically/biologically meaning. Therefore, we take different approach. Before that we recalling the derivation of \eqref{rel heat equation}. To obtain this equation, different methods have been exploited before in the literature. For instance, \cite{Zakari97} used continued-fraction technique to derive \eqref{rel heat equation} from the hierarchy of hydrodynamic equations. As another example, in the context of the flux-limited chemotaxis modelling, in \cite{BBNS10} the authors obtained it by optimizing the the flux density of particles along the trajectory induced by the chemoattractant. Note that in Eq. \eqref{rel heat equation}, the drift term is equal to $\div[u\nabla \varphi^*_c(\nabla V)]$. Motivated by this observation, we now include the external force into \eqref{eq: eqn without force field} simply by generalising the above term to $\div[u\nabla \varphi^*(\nabla V)]$ and adding it to \eqref{eq: eqn without force field}. This results in the following equation
\begin{equation}
\partial_t u=\div[u\nabla\varphi^*(\nabla V)]+\div[u\nabla \varphi^*(\nabla \log u)],
\end{equation}
which is exactly Eq. \eqref{eq: gRHE eqn} that we started with. When $\varphi=\varphi_c$, we do get Eq. \eqref{rel heat equation}. One important property of both Eq. \eqref{rel heat equation} and the classical linearly drift-diffusion equation is that the Gibbs measure $Z^{-1}\exp(-V)$ is the unique equilibrium solution.  In Lemma \ref{lem: Gibbs} below, we prove that this property is preserved in Eq. \eqref{eq: gRHE eqn}. As an additional observation, since $\varphi^*_c(z)\to z$ as $c\to \infty$, at least formally we recover the classical drift-diffusion equation from Eq. \eqref{rel heat equation} as $c\to \infty$. It would be interesting to investigate whether one could derive Eq. \eqref{eq: gRHE eqn} solely from the optimal transport theory as in \cite{CP09} and recover its classical counterpart as $c$ goes to infinity . We provide more discussion about this in Section \ref{sec: conclusion}.

We now show the following. 

\begin{lemma} 
\label{lem: Gibbs}
Eq. \eqref{eq: stationary eqn} has an equilibrium solution which is given by
\begin{equation}
\label{eq: Botzmann dis}
u_{eq}(x)=C\exp(-V(x)),
\end{equation}
for some constant $C$. If $V$ satisfies that $\int \exp(-V(x))\,dx<\infty$, $u_{eq}$ can be normalized to become the Gibbs (probability) measure.
\end{lemma}
\begin{proof}
The proof will use the following properties of $\nabla\varphi^*$:
\begin{enumerate}[(P1)]
\item $\nabla\varphi^*$ is odd, i.e., $\nabla\varphi^*(z)=-\nabla\varphi^*(-z)$. This is because $\varphi^*(z)=f(|z|)$ for some function $f$.
\item $\nabla\varphi^*$ is injective, i.e., if $\nabla\varphi^*(z_1)=\nabla\varphi^*(z_2)$ then $z_1=z_2$. This is because $\varphi^*$ is a convex function that implies that $\nabla\varphi^*$ is monotone.
\end{enumerate}
We recall that $u$ is called an equilibrium solution of Eq.~\eqref{eq: gRHE eqn} if it annihilates the total flux. Now, suppose that $u_{eq}$ is an equilibrium solution, i.e.,
\begin{equation*}
\nabla\varphi^*(\nabla V)+\nabla \varphi^*(\nabla \log u_{eq})=0.
\end{equation*}
Hence by (P1)
\begin{equation*}
\nabla \varphi^*(\nabla \log u_{eq})=-\nabla\varphi^*(\nabla V)=\nabla\varphi^*(-\nabla V),
\end{equation*}
and then by $(P2)$, 
\begin{equation*}
\nabla \log u_{eq}=-\nabla V,
\end{equation*}
which implies that $u_{eq}=C\exp(-V)$ for some constant $C$. If $\int\exp(-V(x))\,dx<\infty$, there exists only one value of $C$, namely $C=\left(\int\exp(-V(x))\,dx\right)^{-1}$, such that $u_{eq}$ is a probability measure on $\R^d$. This is exactly the Gibbs measure. 

However, we note that Eq. \eqref{eq: gRHE eqn} is still meaningful without this extra condition. This is also true for the classical drift-diffusion equation. In addition, its Wasserstein gradient flow structure is rigorously well-defined without the extra assumption on the growth of $V$, see \cite{JKO98}.
\end{proof}

\section{The stationary problem}
\label{sec: eliptic}
In this section we deal with the stationary problem \eqref{eq: evolutionary prob}, i.e.,
\begin{equation}
\label{eq: stationary eqn}
Lu=0,
\end{equation}
where the operator $L$ is defined in \eqref{eq: operator L}.

We consider a change of variable $w=\log u$, and transform Eq. \eqref{eq: stationary eqn} into the following one
\begin{equation}
\label{eq: stationary w}
Qw=0,
\end{equation}
where the operator $Q$ is defined by
\begin{equation}
\label{eq: Q}
Qw=\div[\nabla\varphi^*(\nabla w)]+\nabla w\cdot\nabla\varphi^*(\nabla w)+\nabla w\cdot \nabla \varphi^*(\nabla V)+\div \nabla \varphi^*(\nabla V).
\end{equation}
The transformation plays a crucial role in the present paper. We provide more discussion about it in Remark \ref{re: comment 1}.

We now show that the stationary problem \eqref{eq: stationary w} satisfies a tangency principle, a strong maximum principle and a comparison principle. Let us begin with the tangency principle.
\begin{proposition}[Tangency principle for the operator $Q$]
\label{prop: tangency Q}
Suppose that $\Omega\subset \R^d$ is open and connected and that $w, w'\in C^2(\Omega)$. If $Q w\geq Q w'$,  $w\leq w'$ in $\Omega$, and $w=w'$ at some point in $\Omega$, then $w=w'$ in $\Omega$.
\end{proposition}
\begin{proof}
We first show that $Q$ is a strictly elliptic operator. Indeed, since
\begin{equation*}
\div[\nabla\varphi^*(\nabla w)]=\sum_{i,j}\frac{\partial^2\varphi^*}{\partial z_i\partial z_j}(\nabla w)\frac{\partial^2 w}{\partial z_i\partial z_j}=\nabla^2\varphi^*(\nabla w):\nabla^2 w,
\end{equation*}
the operator $Q$ is equal to
\begin{equation}
\label{eq: operator Q}
Qw=\nabla^2\varphi^*(\nabla w):\nabla^2 w+\nabla w\cdot\nabla\varphi^*(\nabla w)+\nabla w\cdot \nabla \varphi^*(\nabla V)+\div \nabla \varphi^*(\nabla V).
\end{equation}
We can write $Q$ in the form of a quasilinear operator studied in the by-now classical monograph \cite[Chapter 10]{GT01}
\begin{equation}
Qw=\sum_{i,j=1}^d a^{ij}(x,w,\nabla w)\partial^2_{ij} w+b(x,w,\nabla w),
\end{equation}
where the coefficients of $Q$, namely the functions $a^{ij}(x,z,p), i,j=1,\ldots,d$ and $b(x,z,p)$ are defined on $\Omega\times \R\times R^d$. Note that the coefficients are independent of $z$: $b$ depends on $x,p$, while $a$ depends only on $p$. More precisely,
\begin{equation}
a(x,z,p)=\nabla^2\varphi^*(p),\qquad b(x,z,p)=p\cdot\nabla\varphi^*(p)+p\cdot \nabla \varphi^*(\nabla V)+\div \nabla \varphi^*(\nabla V).
\end{equation}
%
Since $\varphi^*$ is strictly convex, the matrix $\nabla^2\varphi^*(z)$ is positive definite. Therefore, $Q$ is a strictly elliptic operator. Furthermore, since the coefficients of $Q$ are independent of $w$ and continuous differentiable in $\nabla w$ and $\nabla^2 w$, the assertion of the proposition is obtained by applying the tangency principle for nonlinear elliptic operators in \cite[Theorem 2.1.3]{PS07}.

\end{proof}
The tangency principle is a strong statement. It allows us to obtain the following strong maximum/minimum principle for the operator $Q$.
\begin{proposition}[Strong maximum principle for the operator $Q$]
\label{prop: max Q}
\ \\
Suppose that $\Omega\subset \R^d$ is open and connected. 
\begin{enumerate}[(1)]
\item Assume that $w\in C^2(\Omega)\cap C(\overline{\Omega})$ and that $\div(\nabla\varphi^*(\nabla V))\leq 0$. If $Qw\geq 0$ in $\Omega$ and $w$ attains its maximum over $\overline{\Omega}$ at an interior point, then $w$ is constant within $\Omega$. 
\item Similarly, assume that $w'\in C^2(\Omega)\cap C(\overline{\Omega})$ and that $\div(\nabla\varphi^*(\nabla V))\geq 0$. If $Qw'\leq 0$ in $\Omega$ and $w'$ attains its minimum over $\overline{\Omega}$ at an interior point, then $w'$ is constant within $\Omega$. 
\end{enumerate}
\end{proposition}
\begin{proof}
\begin{enumerate}[(1)]
\item In Proposition \ref{prop: tangency Q}, we take $w'=M$, which is the maximum attained by $w$. We have  
\begin{equation*}
Qw'=\div(\nabla \varphi^*(\nabla V))\leq 0\leq Qw.
\end{equation*}
In addition, $w\leq M$, and $u=M$ is attained at some point in $\Omega$. The strong maximum principle is then followed straightforwardly from the tangency principle.

\item Similarly the minimum principle is proved by taking $w=m'$, which is the minimum attained by $w'$. 
\end{enumerate}
\end{proof}
\begin{remark}
\label{re: computation}
In the case of the relativistic heat equation \eqref{rel heat equation}, with
(for simplicity we put $c=1$)
\begin{equation}
\varphi_c^*(x)=\sqrt{1+|x|^2}-1,
\end{equation}
we have
\begin{equation*}
 \nabla \varphi_c^*(x)=\frac{x}{\sqrt{1+|x|^2}},\quad\text{and}\quad(\nabla^2\varphi^*(x))_{ij}=\frac{1}{\sqrt{1+|x|^2}}\left(\delta_{ij}-\frac{x_ix_j}{1+|x|^2}\right).
\end{equation*}
Hence, we obtain
\begin{align*}
Qw&=\nabla^2\varphi^*(\nabla w):\nabla^2 w+\nabla w\cdot\nabla\varphi^*(\nabla w)+\nabla w\cdot \nabla \varphi^*(\nabla V)+\div \nabla \varphi^*(\nabla V).
\\&=
\frac{\Delta w(1+|\nabla w|^2)-\nabla^2 w:(\nabla w\otimes \nabla w)}{(1+|z|^2)^{\frac{3}{2}}}+ 
\frac{|\nabla w|^2}{\sqrt{1+|\nabla w|^2}}+\frac{\nabla w\cdot\nabla V}{\sqrt{1+|\nabla V|^2}}+
\\&+\qquad\frac{\Delta V(1+|\nabla V|^2)-\nabla^2V:(\nabla V\otimes\nabla V)}{(1+|\nabla V|^2)^\frac{3}{2}}.
\\&=\sum_{i,j=1}^d a^{ij}(x,w,\nabla w)\partial^2_{ij} w+b(x,w,\nabla w),
\end{align*}
where
\begin{align*}
&a_{ij}(x,z,p)=(\nabla^2\varphi^*(p))_{ij}=\frac{1}{\sqrt{1+|p|^2}}\left(\delta_{ij}-\frac{p_ip_j}{1+|p|^2}\right),
\\& b(x,z,p)=\frac{|p|^2}{\sqrt{1+|p|^2}}+\frac{p\cdot\nabla V}{\sqrt{1+|\nabla V|^2}}+\frac{\Delta V(1+|\nabla V|^2)-\nabla^2V:(\nabla V\otimes\nabla V)}{(1+|\nabla V|^2)^\frac{3}{2}}
\end{align*}
It is interesting to note that in this case, the principle part of $Q$ is exactly the mean-curvature operator. This raises some interesting questions, which we comment further in Section \ref{sec: conclusion}.

In particular, in the one-dimensional setting $d=1$, we have
\begin{equation}
\div(\nabla \varphi^*(\nabla V))=\frac{V''(1+V'^2)-V'' V'^2}{(1+V'^2)^\frac{3}{2}}=\frac{V''}{(1+V'^2)^\frac{3}{2}}.
\end{equation}
Then the condition on $V$ in the maximum/minimum principle in Proposition \ref{prop: max Q} is easily verified: $\div\nabla \varphi^*(\nabla V)\leq 0~ (\geq 0)$ if and only if $V$ is concave (convex, respectively).
\end{remark}

We proceed with a comparison principle for the operator $Q$ in the following proposition.
\begin{proposition}[Comparison principle for the operator $Q$]
\label{prop: comparison Q}
Suppose $w,w'\in C^2(\Omega)\cap C(\overline{\Omega})$ and satisfy $Qw\geq Qw'$ in $\Omega$. If $w\leq w'$ on $\partial \Omega$, then $w\leq w'$ in $\Omega$.
\end{proposition}
\begin{proof}
Similarly as in the proof of Proposition \ref{prop: tangency Q}, this proposition is obtained by applying the comparison for nonlinear elliptic operators in \cite[Theorem 2.1.4]{PS07}.
\end{proof}

We are now ready to provide the proofs of the main results for the original stationary problem.
\begin{proof}[\textbf{Proofs of Theorem \ref{theo: tangency L}, Theorem \ref{theo: strong max L} and Theorem \ref{theo: comparison L}}]
\ \\
\ \\
Since the exponential transformation $u=\exp(w)$ is positive and monotone, the statements of  Theorem \ref{theo: tangency L}, Theorem \ref{theo: strong max L} and Theorem \ref{theo: comparison L} respectively follow from those of Proposition \ref{prop: tangency Q}, Proposition \ref{prop: max Q} and Proposition \ref{prop: comparison Q}.
\end{proof}
We now discuss more on the transformation between $u$ and $w$.
\begin{remark}
\label{re: comment 1}
As one can recognise from the proofs in this section, the transformation $w=\log u$ has two advantages. Firstly, it makes the operator $Q$ depends only on $\nabla w$ and not on $w$ itself. This enables us to obtain maximum and comparison principles for the operator $Q$ by applying known results for quasilinear elliptic operators. Secondly, it is positive and monotone, therefore, we achieve results for $P$ straightforwardly from the corresponding ones for $Q$.

The transformation is clearly hinted by the form of Eq. \eqref{eq: gRHE eqn}. Tracing back from its derivation in Section \ref{sec: derivation}, it is the gradient flow structure of the Boltzmann entropy with respect to the Kantorovich optimal transport functional associated to the general relativistic cost that gives rise to that form. In other words, the transformation matches perfectly well with the structure of the equation.

The exponential transformation has been frequently used in the literature for the study of nonlinear partial differential equations (PDEs). However, note that the exponential transformation in this paper, is different from the well-known Cole-Hopf transformation. The latter is used to convert a nonlinear into a linear PDE. While we still get a nonlinear PDE in the former. If $\varphi$ is the classical cost function, $\varphi(x)=\frac{x^2}{2}$, the transformation in this paper is actually inverse of the Cole-Hopf one.
\end{remark}
\section{The evolutionary problem}
\label{sec: parabolic}
In this section, we prove the main results, Theorem \ref{theo: comparison L evol} and Theorem \ref{theo: weak max L evol} for the evolutionary problem \ref{eq: evolutionary prob}
\begin{equation*}
\partial_t u=Lu.
\end{equation*}
Using the transformation as in the stationary case, the above equation is transformed to
\begin{equation}
\partial_t w=Qw.
\end{equation}
The strategy now is analogous to the stationary case in the previous section. We prove comparison and maximum principles for the evolutionary problem associated to the operator $Q$ using known theories in the literature, and then obtain the corresponding results for the original problem for the operator $L$ using the positivity and monotonicity of the transformation.

We recall that the definition on the spacetime domain and functions are given in the Introduction. We first consider a comparison principle for the evolutionary problem associated to $Q$.
\begin{proposition}[Comparison principle for the evolutionary problem for $Q$]
\  \\
\label{prop: comparison Q evol}
Suppose that $w,w'\in C^{2,1}(\Omega_T)\cap C(\overline{\Omega_T})$ satisfying $\partial_t w-Q w\leq \partial_t w'_t-Qw'$ in $\Omega$. Then $\max\limits_{\overline{\Omega_T}}(w-w')=\max\limits_{\Gamma_T}(w-w')$. In particular, if $w\leq w'$ on $\Gamma_T$, then $w\leq w'$ in $\Omega_T$.
\end{proposition}
\begin{proof}
The proof follows the procedure as in \cite[Theorem 3.3]{MS15TR}, which is in turn mainly adapted from \cite[Theorem 10.1]{GT01}. We rearrange the inequality in the assumption to obtain
\begin{equation}
\label{eq: evol enequality}
\partial_t(w-w')-(Qw-Qw')\leq 0.
\end{equation}
Set $\overline{b}(p)=p\cdot \nabla\varphi^*(p)+p\cdot \nabla\varphi^*(\nabla V)$. We have
\begin{align*}
Qw-Qw'&=\mathbf{a}(\nabla w)\nabla^2 w+\tilde{b}(\nabla w)-\mathbf{a}(\nabla w')\nabla^2 w'-\overline{b}(\nabla w')
\\&=\mathbf{a}(\nabla w)(\nabla^2 w-\nabla^2 w')+(\mathbf{a}(\nabla w)-\mathbf{a}(\nabla w'))\nabla^2 w'+\tilde{b}(\nabla w)-\overline{b}(\nabla w').
\end{align*}
We set $z=w-w'$, and define the linear operator $\tilde{Q}z=\mathbf{a}(x,t)\nabla^2 z+\tilde{b}(x,t)\nabla z$, with
\begin{align*}
&\mathbf{a}(x,t)=\mathbf{a}(\nabla w),
\\&\tilde{b}(x,t)\nabla z=(\mathbf{a}(\nabla w)-\mathbf{a}(\nabla w'))\nabla^2 w'+\overline{b}(\nabla w)-\overline{b}(\nabla w').
\end{align*}
Note that the existence of the function $b$ is due to the mean value theorem. Then we obtain $Qw-Qw'=\tilde{Q}z$, and \eqref{eq: evol enequality} becomes
\begin{equation*}
\partial_t z-\tilde{Q}z\leq 0,
\end{equation*}
in $\Omega_T$. In addition, $z\leq 0$ on $\Gamma_T$. By the parabolic weak maximum principle \cite[Theorem 8, Section 7]{Evans98}, we have that $z\leq 0$ in $\Omega_T$, i.e., $w\leq w'$ in $\Omega_T$. This completes the proof.
\end{proof}
As a consequence of the comparison principle we obtain the weak maximum principle for evolutionary problem.
\begin{proposition}[Weak maximum principle for the evolutionary problem for $Q$]
\ \\
The following statements hold.
\label{prop: weak max Q evol}
\begin{enumerate}[(1)]
\item Suppose that $w\in C^{2,1}(\Omega_T)\cap C(\overline{\Omega_T})$ and $\div(\nabla\varphi^*(\nabla V))\geq 0$. If $\partial_t w-Qw \leq 0$, then $\max\limits_{\overline{\Omega_T}} w=\max\limits_{\Gamma_T}w$.
\item Similarly, suppose that $w'\in C^{2,1}(\Omega_T)\cap C(\overline{\Omega_T})$ and $\div(\nabla\varphi^*(\nabla V))\leq 0$. If $\partial_t w'-Qw'\geq 0$ , then $\min\limits_{\overline{\Omega_T}} w'=\min\limits_{\Gamma_T}w'$.
\end{enumerate}
\end{proposition}
\begin{proof}
\begin{enumerate}[(1)]
\item This statement follows directly from Proposition \ref{prop: comparison Q evol} by taking $w'=0$.
\item Similarly, this statement follows directly from Proposition \ref{prop: comparison Q evol} by taking $w=0$.
\end{enumerate}
\end{proof}
Now we are ready to prove the main results for the original operator $L$.
\begin{proof}[\textbf{Proofs of Theorem \ref{theo: comparison L evol} and Theorem \ref{theo: weak max L evol}}]
\ \\
\ \\
Again, by the positivity and monotonicity of the exponential transformation $u=\exp(w)$, Theorem \ref{theo: comparison L evol} and Theorem \ref{theo: weak max L evol} respectively follow from Proposition \ref{prop: comparison Q evol} and Proposition \ref{prop: weak max Q evol}.
\end{proof}
\section{Conclusion and outlook}
\label{sec: conclusion}
In this paper, we derive a generalised relativistic heat equation with external force fields and study comparison and maximum principles for this equation. The latter has been carried out using an logarithmic transformation which hinges on the gradient flow structure of the Boltzmann entropy with respect to an optimal transport associated to the general relativistic cost function. The derivation opens some interesting problems. The first possibility is to derive Eq. \eqref{eq: gRHE eqn} use only optimal transport theory. It is not clear which energy functional should be used. The second open question is to rigorously prove that when $c$ goes to infinity, Eq.\eqref{rel heat equation} degenerates to the classical linear drift-diffusion equation. Research in this direction has recently been received lot of attention. For example, in \cite{Cas07, Cas15}, the author has proved that solutions of the relativistic heat equation without external force field indeed converge to that of the classical heat equation. We believe that the methods in these papers can be extended to include the external force fields. The third   open problem is to derive Eq. \eqref{eq: gRHE eqn}, and in general flux-limited diffusions, from stochastic models. Connections between partial differential equations and stochastic processes 
often provide insightful information about the geometrical structure and also offer new techniques for coarse-graining of the former, see e.g. \cite{DPZ13a, DPZ13b, DLR13, DLPS15} for more information in this direction. As mentioned in Remark \ref{re: computation}, the principle part of the operator $Q$ is exactly the well-known mean curvature operator. Derivation of the mean-curvature flow using stochastic approach has been studied extensively in the literature, for instance in \cite{ST03, Soner07}. It would be interesting to investigate whether one can derive Eq. \eqref{eq: FP eqn}, and in general flux-limited diffusions, using such approaches. We leave these open problems for further research in the future.
\bibliographystyle{abbrv}
\bibliography{bibGENERIC}
\end{document}